\documentclass[11pt,a4paper,twoside]{amsart}

\usepackage{amsthm, amsmath, enumerate,array}
\usepackage{hyperref,cite,mdwlist}

\usepackage[T1]{fontenc}
\usepackage[utf8]{inputenc}

\usepackage[all]{xy}
\usepackage{subfigure}
\usepackage{graphicx}
\usepackage{pgf,tikz}
\usetikzlibrary{matrix,calc,arrows,decorations.markings}

\usepackage[english]{babel}
\usepackage[bitstream-charter]{mathdesign}

\usepackage[left=1.5in,top=1.2in,right=1.5in,bottom=1.22in]{geometry}

\newtheorem{thm}{Theorem}
\newtheorem{lem}[thm]{Lemma}

\newtheorem*{thm*}{Theorem}
\newtheorem*{cnj*}{Conjecture}

\theoremstyle{definition}

\newtheorem{rmk}[thm]{Remark}

\newcommand{\cH}{\mathcal{H}}

\newcommand{\XE}{E}

\newcommand{\cO}{\mathcal{O}}

\newcommand{\bk}{{\boldsymbol{k}}}

\DeclareMathOperator{\rk}{rk}

\DeclareMathOperator{\Ext}{Ext}
\DeclareMathOperator{\ext}{ext}
\DeclareMathOperator{\Hom}{Hom}
\DeclareMathOperator{\RHom}{RHom}

\DeclareMathOperator{\HH}{H}
\DeclareMathOperator{\hh}{h}

\DeclareMathOperator{\coker}{coker}

\DeclareMathOperator{\rL}{L}
\DeclareMathOperator{\rR}{R}

\newcommand{\ZZ}{\mathbb Z}

\newcommand{\PP}{\mathbb P}

\DeclareMathOperator{\ts}{\otimes}

\newcommand{\epi}{\twoheadrightarrow}
\newcommand{\xr}{\xrightarrow}

\usepackage{xcomment}

\allowdisplaybreaks[1]

\begin{document}

\sloppy

\addtocontents{toc}{\protect\setcounter{tocdepth}{2}}


\title{Yet again on two examples by Iyama and Yoshino}


\author{Daniele Faenzi}
\email{\tt daniele.faenzi@univ-pau.fr}
\address{Université de Pau et des Pays de l'Adour \newline
  \indent Avenue de l'Université - BP 576 \newline
  \indent 64012 PAU Cedex - France}
\urladdr{{\url{http://univ-pau.fr/~faenzi/}}}

\thanks{Author partially supported by  ANR GEOLMI ANR-11-BS03-0011}


\maketitle

\begin{abstract}
  We give an elementary proof of Iyama-Yoshino's classification of
  rigid MCM modules on Veronese embeddings in $\PP^9$.
\end{abstract}

\section*{Introduction}

The beautiful theory of cluster tilting in triangulated categories has been
developed by Iyama and Yoshino; as an important outcome of this the authors gave in \cite[Theorem 1.2 and Theorem
1.3]{iyama-yoshino} the classification of rigid indecomposable MCM modules over two
Veronese embeddings in $\PP^9$ given, respectively, by plane cubics and space
quadrics. 
Another proof, that makes use of Orlov's singularity
category, appears in \cite{keller-murfet-van_den_bergh}, where the link
between power series Veronese rings and the graded rings of the
corresponding varieties is also explained.
Also, \cite{keller-reiten:acyclic} contains yet another argument.

The goal of this note is to present a simple proof of Iyama-Yoshino's
classification of rigid MCM modules over the aforementioned Veronese
rings, making use of
vector bundles and Beilinson's theorem.
This proof works over a field $\bk$ which is algebraically
closed or finite.


Consider the embedding of the projective space $\PP^n$ given by
homogeneous forms of degree $d$, i.e. the {\it $d$-fold
  Veronese variety}.
A coherent sheaf $E$ on $\PP^n$ is arithmetically Cohen-Macaulay (ACM)
with respect to this embedding if and only 
if $E$ is locally free and has {\it no intermediate cohomology}:
\begin{equation}
  \label{sonmalato}
\HH^i(\PP^n,E(d t))=0, \qquad \mbox{for all $t \in \ZZ$ and all $0<i<n$}.  
\end{equation}
This is equivalent to ask that the module of global sections
associated with $E$ is MCM over the corresponding Veronese ring.
For $d$-fold Veronese embeddings of $\PP^n$ in $\PP^9$ (i.e. $\{n,d\}=\{2,3\}$), we are going to classify ACM
bundles $E$ which are {\it rigid}, i.e. $\Ext^1_{\PP^n}(E,E)=0$.
We set $\ell = {n+1 \choose 2}$. 

To state the classification, 
 we define
the {\it Fibonacci numbers} $a_{\ell,k}$ by the relations: 
$a_{\ell,0}=0$, $a_{\ell,1}=1$ and $a_{\ell,k+1}=\ell a_{\ell,k}-a_{\ell,k-1}$.
For instance $(a_{3,k})$ is given by the odd values of the
usual Fibonacci sequence:
\[
  a_{3,k}=0,1,3,8,21,55,144,\ldots \qquad \mbox{for $k=0,1,2,3,4,5,6,\ldots$} 
\]

\begin{thm}[\{n,d\}=\{2,3\}] \label{class2}
Let $E$ be an indecomposable bundle on $\PP^n$ satisfying
\eqref{sonmalato}.
\begin{enumerate}[i)]
\item \label{perforzasteiner} If $E$ has no endomorphism factoring
  through $\cO_{\PP^n}(t)$, then there are
  $a,b \ge 0$ such that, up to a twist by 
$\cO_{\PP^n}(s)$, $E$ or $E^*$ is the cokernel of an injective map:
\[
\Omega_{\PP^n}^2(1)^b \to \cO_{\PP^n}(-1)^a
\]
\item \label{quandoerigido} If $E$ is indecomposable and rigid, then there is $k\ge 1$
  such that, up to tensoring with $\cO_{\PP^n}(s)$, $E$ or $E^*$ is the cokernel of an injective map:
\[
\Omega_{\PP^n}^2(1)^{a_{\ell,k-1}} \to \cO_{\PP^n}(-1)^{a_{\ell,k}}
\]
\item \label{nienteH1} Conversely for any $k \ge 1$, there is a unique indecomposable
  bundle $E_k$ having a resolution of the form:
  \[
  0 \to \Omega_{\PP^n}^2(1)^{a_{\ell,k-1}} \to \cO_{\PP^n}(-1)^{a_{\ell,k}} \to E_k \to 0,
  \]
  and both $E_k$ and $E^*_k$ are ACM and exceptional.
\end{enumerate}
\end{thm}

In the previous statement, it is understood that a bundle $E$ is 
{\it  exceptional} if it is rigid, {\it simple}
(i.e. $\Hom_{\PP^n}(E,E) \simeq \bk$) and $\Ext^i_{\PP^n}(E,E)=0$ for $i \ge 2$.
Next, we write $\Omega_{\PP^n}^p=\wedge^p \Omega_{\PP^n}^1$ for the bundle
of differential $p$-forms on $\PP^n$.

\begin{rmk} Part \eqref{perforzasteiner} of Theorem \ref{class2}
  is a version of Iyama-Yoshino's  general results on Veronese
  rings \cite[Theorem 9.1 and 9.3]{iyama-yoshino}, to the
  effect that for $\{n,d\}=\{2,3\}$ the stable category of MCM modules is equivalent to
  the category of representations of a certain Kronecker quiver.
  However our result is  algorithmic, for it provides
  the representation associated with an MCM
  module via Beilinson spectral sequence applied to the corresponding ACM bundle.
\end{rmk}

\begin{rmk}
  The rank of the bundle $E_k$ is given by the Fibonacci number between $a_{3,k-1}$ and
  $a_{3,k}$ in case $(n,d)=(2,3)$.
  In this case $E_{2k}$ (respectively, $E_{2k+1}$) is the $k$-th
  sheafified syzygy occurring in the 
  resolution of $\cO_{\PP^2}(1)$ (respectively, of $\cO_{\PP^2}(2)$)
over the Veronese ring, twisted by 
$\cO_{\PP^2}(3(k-1))$. A similar result holds for $(n,d)=(3,2)$.
\end{rmk}

As for notation, we write small letters for the dimension
of a space in capital letter, for instance $\hh^i(\PP^n,E)=\dim_\bk \HH^i(\PP^n,E)$.
We also write $\chi(E,F)=\sum(-1)^i \ext^i_{\PP^n}(E,F)$ and
$\chi(E)=\chi(\cO_{\PP^n},E)$. $\delta_{i,j}$ is Kronecker's delta.

\section{Fibonacci bundles}

\subsection{}
  Let us write
  $\Upsilon_\ell$ for the $\ell$-th Kronecker quiver, namely the oriented
  graph with two vertices $\mathbf{e_0}$ and $\mathbf{e_1}$, and $\ell$
  arrows from $\mathbf{e_0}$ to $\mathbf{e_1}$.
  A representation $R$ of  $\Upsilon_\ell$, with
  dimension vector $(a,b)$ is the choice of $\ell$ matrices of size $a
  \times b$. 
  \[\begin{tikzpicture}[scale=1]
    \draw (-2.5,0) node [] {$\Upsilon_3$:}; 
    \draw (-1,0) node [above] {$ \mathbf{e_1}$};
    \draw (1,0) node [above] {$\mathbf{e_2}$};
    \node (1) at (-1,0) {$\bullet$};
    \node (1) at (-1,0) {$\bullet$};
    \node (2) at (1,0) {$\bullet$};
    \draw[->,>=latex] (1) to (2);
    \draw[->,>=latex] (1) to[bend left] (2);
    \draw[->,>=latex] (1) to[bend right] (2);
  \end{tikzpicture}
  \]
  We identify a basis of $\HH^0(\PP^n,\Omega_{\PP^n}(2))$ with the set
  of $\ell={n+1 \choose 2}$ arrows of $\Upsilon_\ell$. 
  Then the derived category of finite-dimensional representations of
  $\Upsilon_\ell$
  embeds into the derived category of $\cO_{\PP^n}$-modules by sending $R$ to
  the cone $\Phi(R)$ of the morphism $e_R$
  associated with $R$ according to this identification:
  \[
  \Phi(R)[-1] \to \cO_{\PP^n}(-1)^{a} \xr{e_R}  \Omega_{\PP^n}(1)^{b},
  \]
  where we denote by $[-1]$ the shift to the right of complexes. It is clear that:
 \[
 \Ext_{\PP^n}^i(\Phi(R),\Phi(R)) \simeq \Ext^i_{\Upsilon_\ell}(R,R),  \qquad \mbox{for  all $i$}.
 \]

\subsection{} \label{KAC}
We will use Kac's classification of rigid $\Upsilon_\ell$-modules as
Schur roots (hence the restriction on $\bk$), which is also one of the main ingredients in
Iyama-Yoshino's proof.
By \cite[Theorem 4]{kac}, any non-zero rigid
$\Upsilon_\ell$-module is a direct sum of rigid simple
representations of the form $R_k$, for some
 $k \in \ZZ$, where $R_k$ is defined as the unique indecomposable
 representation of $\Upsilon_\ell$ with dimension vector
 $(a_{\ell,k-1},a_{\ell,k})$ for $k \ge 1$, or  $(a_{\ell,-k},a_{\ell,1-k})$
 for $k \le 0$.

Set $F_k=\Phi(R_k)$ for $k \ge 1$, and $F_k=\Phi(R_k)[-1]$ for $k \le 0$.
It turns out that $F_k$ is an exceptional locally free sheaf, called a
{\it Fibonacci bundle}, cf. \cite{brambilla:fibonacci}.
We rewrite the defining exact sequences of $F_k$:
\begin{align}
  \label{defining+}
  &0 \to \cO_{\PP^n}(-1)^{a_{\ell,k-1}} \to  \Omega_{\PP^n}(1)^{a_{\ell,k}}  \to F_k \to  0, && \mbox {for $k \ge 1$},\\
 \nonumber &0 \to F_k \to \cO_{\PP^n}(-1)^{a_{\ell,-k}} \to \Omega_{\PP^n}(1)^{a_{\ell,1-k}}  \to 0, && \mbox {for $k \le 0$}.
\end{align}

\subsection{} Here is a lemma on the cohomology of Fibonacci bundles.

\begin{lem} \label{cohom}
For $k \ge 1$, the only non-vanishing intermediate cohomology of $F_k$
is:
\[
\hh^1(\PP^n,F_k(-1))=a_{\ell,k-1} \qquad \hh^{n-1}(\PP^n,F_k(-n))=a_{\ell,k}.
\]
\end{lem}

\begin{proof}
  We consider the left and right mutation
  endofunctors of the derived category of coherent sheaves on $\PP^n$,
  that associate with a pair $(E,F)$ of complexes, two complexes denoted respectively by $\rR_FE$ and $\rL_E F$. These are the cones of the natural
  evaluation maps $f_{E,F}$ and $g_{E,F}$:
  \[
   E \xr{f_{E,F}} \RHom_{\PP^n}(E,F)^* \ts F \to \rR_FE, \qquad  
  \rL_E F \to \RHom_{\PP^n}(E,F) \ts E \xr{g_{E,F}} F.
  \]
  
  It is well-known (cf. \cite{brambilla:fibonacci}) that the Fibonacci bundles $F_k$ can be defined recursively
  from $F_0=\cO_{\PP^n}(-1)$ and $F_1=\Omega_{\PP^n}(1)$ by setting:
  \begin{align}
    \label{+} &F_{k+1}=\rR_{F_k} F_{k-1}, && \mbox{for $k \ge 1$},    \\
    \label{-} &F_{k-1}=\rL_{F_k} F_{k+1}, && \mbox{for $k \le 0$}.
  \end{align}
  This way, for any $k\in \ZZ$ we get a natural exact sequence:
  \begin{equation}
    \label{recurrence}
  0 \to F_{k-1} \to (F_k)^\ell \to F_{k+1} \to 0.    
  \end{equation}

Over $\PP^n$, we consider the full exceptional
sequence:
\[
(\cO_{\PP^n}(-1),\Omega_{\PP^n}(1),\cO_{\PP^n},\cO_{\PP^n}(2),\cO_{\PP^n}(3),\ldots,\cO_{\PP^n}(n-1)),
\]
obtained from the standard collection $(\cO_{\PP^n}(-1),\ldots,\cO_{\PP^n}(n-1))$,
by the mutation $\Omega_{\PP^n}(-1) \simeq
\rL_{\cO_{\PP^n}}\cO_{\PP^n}(1)$ (all the terminology and results we
need on exceptional collections are contained in
\cite{bondal:representations-coherent}).
By \eqref{+}, we can replace the previous
exceptional sequence with:
\[
(F_{k-1},F_{k},\cO_{\PP^n},\cO_{\PP^n}(2),\ldots,\cO_{\PP^n}(n-1)),
\]
Right-mutating $F_{k-1}$ through the full collection, we must get back
$F_{k-1} \ts \omega_{\PP^n}^* \simeq F_{k-1}(n+1)$.
So, using \eqref{recurrence}, we get a long exact sequence:
\begin{equation}
  \label{lunga}
0 \to F_{k+1} \to \cO_{\PP^n}^{u_{1}} \to
\cO_{\PP^n}(2)^{u_{2}} \to \cdots \to \cO_{\PP^n}(n-1)^{u_{n-1}}
\to F_{k-1}(n+1) \to 0,
\end{equation}
for some integers $u_{i}$.
Now by \eqref{defining+} we get:
\[
\HH^i(\PP^n,F_k(t))=0 \quad \mbox {for:} \quad
\left\{
  \begin{array}[h]{ll}
    2 \le i \le n-2, & \forall t, \\
    i = 0, & t \le 0, \\
    i = 1, & t \ne -1, \\
    i = n-1, & t \ge 1-n.
  \end{array}
\right.
\]
The required non-vanishing cohomology of $F_k$ appears again from
\eqref{defining+}.
So it only remains to check that $\HH^{n-1}(\PP^n,F_k(t))=0$ for $t
\le -n-1$. But this is clear by induction once we twist \eqref{lunga} by
$\cO_{\PP^n}(t)$, and take cohomology.
\end{proof}

\subsection{} We compute the Ext groups
between pairs of Fibonacci bundles.

\begin{lem} \label{reciproci}
  For any pair of integers $j \ge k+1$ we have:
  \[
  \ext^i_{\PP^n}(F_j,F_k) = \delta_{1,i} a_{\ell,j-k-1}, \qquad   \ext^i_{\PP^n}(F_k,F_j) = \delta_{0,i} a_{\ell,j-k+1}.
  \]
\end{lem}

\begin{proof}
  The formulas hold for $k=j$ since $F_k$ is exceptional, and
  we easily compute $\chi(F_j,F_{k})=-a_{\ell,j-k-1}$ and
  $\chi(F_k,F_{j})=a_{\ell,j-k+1}$ (for instance by computing $\chi$ of
  $\Upsilon_\ell$-modules via the Cartan form and using faithfullness of
  $\Phi$).
 
  The second formula is proved once we show
  $\Ext^i_{\PP^n}(F_k,F_j) = 0$ for $i \ge 1$.
  In fact, since the category of
  $\Upsilon_\ell$-representations is hereditary, the second formula
  holds if $k \le 0$ for in this case $F_k \simeq \Phi(R_k)[-1]$. 
  By the same reason, we only have to check it for $i=1$.
  Using \eqref{recurrence}, this vanishing holds for $j$ if it does
  for $j-1$ and $j-2$. Since the statement is clear when extended to $j=k$, it
  suffices to check $\Ext^1_{\PP^n}(F_k,F_{k-1}) = 0$.
  Since $\chi(F_k,F_{k-1})=0$, $\Hom_{\PP^n}(F_k,F_{k-1}) = 0$ will do
  the job. However, any nonzero map $F_k \to F_{k-1}$ would give,
  again by \eqref{recurrence}, a non-scalar endomorphism of $F_k$,
  which cannot exist since $F_k$ is simple. The second formula is now proved.

  As for the first formula, again we see that
  it holds if $k \le 0$ and $j \ge 1$ once we check it for
  $i=0$. However using repeatedly \eqref{recurrence} we see that a
  non-zero map $F_j \to F_k$ leads to an endomorphism of
  $F_j$ which factors through $F_k$ this is absurd for $F_j$ is simple.
  When $j,k$ have the same sign, the first formula has to be checked for
  $i=2$ only.
  Moreover, we have just proved the statement for $j=k+1$, and using
  \eqref{recurrence} and exceptionality of $F_k$ we get it for
  $j=k+2$. By iterating this argument we get the statement for any $j
  \ge k+1$.
\end{proof}

\begin{rmk}
  This lemma holds more generally (with the same proof) for any exceptional pair
  $(F_0,F_1)$ of objects on a projective $\bk$-variety $X$, with
  $\hom_X(F_0,F_1)=\ell$, by defining recursively $F_k$ for all $k \in
  \ZZ$ by \eqref{+} and \eqref{-}.
\end{rmk}

\section{Rigid ACM bundles on the third Veronese surface}

We prove here Theorem \ref{class2} in case $(n,d)=(2,3)$.

\subsection{}
 Let us first prove \eqref{perforzasteiner}. So let $E$ be an
 indecomposable vector 
  bundle on $\PP^2$ satisfying \eqref{sonmalato}. 
  Without loss of generality, we may replace $E$ by $E(s)$, where $s$ is the smallest integer such that
  $\hh^0(\PP^2,E(s)) \ne 0$.
  Set $\alpha_{i,j}=\hh^i(\PP^2,\XE(-j))$.   Of course, $\alpha_{0,j}=0$ if and only if $j \ge 0$.
  The Beilinson complex $F$ associated with $\XE$ (see for instance \cite[Chapter 8]{huybrechts:fourier-mukai}) reads: 
  \[
  0
  \to 
   \cO_{\PP^2}(-1)^{\alpha_{1,2}}
  \xr{d_0}
  \begin{array}{c}
    \cO_{\PP^2}(-1)^{\alpha_{2,2}} \\
    \oplus \\
    \Omega_{\PP^2}(1)^{\alpha_{1,1}}  \\
    \oplus \\ 
   \cO_{\PP^2}^{\alpha_{0,0}}
  \end{array}
    \xr{d_1}
  \begin{array}{c}
    \Omega_{\PP^2}(1)^{\alpha_{2,1}}\\
    \oplus \\ 
   \cO_{\PP^2}^{\alpha_{1,0}}
  \end{array}
    \xr{d_2}
   \cO_{\PP^2}^{\alpha_{2,0}}
   \to 0.
  \]
  The term consisting of three summands in the above complex sits in
  degree $0$ (we call it {\it middle term}), and the 
  cohomology of this complex is $\XE$.
  By condition \eqref{sonmalato}, at least one of the $\alpha_{1,j}$ is
  zero, for $j=0,1,2$.

  If $\alpha_{1,2}=0$, then $d_0=0$. By minimality of the Beilinson complex the
  restriction of $d_1$ to the summand $\cO_{\PP^2}^{\alpha_{0,0}}$ of the
  middle term is also zero.
  Therefore $\cO_{\PP^2}^{\alpha_{0,0}}$ is a direct summand of $\XE$, so $\XE
  \simeq \cO_{\PP^2}$ by indecomposability of $E$.

  If $\alpha_{1,1}=0$, then the non-zero component of $d_0$ is just a map
  $\cO_{\PP^2}(-1)^{\alpha_{1,2}} \to \cO_{\PP^2}^{\alpha_{0,0}}$, and a direct
  summand of $\XE$ is the cokernel of this map.
  By indecomposability of $E$, in this case $\XE(-1)$ has a resolution
  of the desired form with $a=\alpha_{0,0}$ and $b=\alpha_{1,2}$.

\subsection{}
It remains to look at the case $\alpha_{1,0}=0$.
  Note that
  the restriction of $d_1$ to  
  $\Omega_{\PP^2}(1)^{\alpha_{1,1}} \oplus  \cO_{\PP^2}^{\alpha_{0,0}}$ is zero,
  which implies that a direct summand of $\XE$ (hence all of $\XE$ by
  indecomposability) has the resolution: 
  \begin{equation}
    \label{generale}
  0 \to    \cO_{\PP^2}(-1)^{\alpha_{1,2}} \xr{d_0} \Omega_{\PP^2}(1)^{\alpha_{1,1}}
  \oplus  \cO_{\PP^2}^{\alpha_{0,0}} \to \XE \to 0
  \end{equation}
  and $\alpha_{2,j}=0$ for $j=0,1,2$.
  We compute $\chi(\XE(-3))=3\alpha_{1,2} - 3\alpha_{1,1}+\alpha_{0,0}$, so:
  \[
  \hh^0(\PP^2,\XE^*)= \hh^2(\PP^2,\XE(-3))= 3\alpha_{1,2} - 3\alpha_{1,1}+\alpha_{0,0}.
  \]
  If this value is positive, then there is a non-trivial morphism
  $g : \XE \to \cO_{\PP^2}$, and since $\alpha_{0,0} \ne 0$ there also
  exists 
  $0 \ne f : \cO_{\PP^2}   \to \XE$. So $\XE$ has an endomorphism
  factoring through $\cO_{\PP^2}$, a contradiction.

  Hence we may assume $3\alpha_{1,2} - 3\alpha_{1,1}+\alpha_{0,0}$, in
  other words $\alpha_{0,3}=0$.
  Therefore, the Beilinson complex associated with $\XE(-1)$ gives a
  resolution:
  \[
  0 \to \XE(-1) \to \Omega_{\PP^2}(1)^{\alpha_{1,2}} \to \cO_{\PP^2}^{\alpha_{1,1}} \to 0.
  \]
  
  It it easy to convert this resolution into the form we want by
  the diagram:
  \[
  \xymatrix@-2.5ex{
    & 0 \ar[d] & 0 \ar[d] \\
    0 \ar[r] & \XE(-1) \ar[r] \ar[d]&  \Omega_{\PP^2}(1)^{\alpha_{1,2}} \ar[r] \ar[d]&  \cO_{\PP^2}^{\alpha_{1,1}} \ar[r] \ar@{=}[d]&  0\\
    0 \ar[r] &\cO_{\PP^2}^{3\alpha_{1,2}-\alpha_{1,1}} \ar[r]\ar[d] &  \cO_{\PP^2}^{3\alpha_{1,2}} \ar[r]\ar[d] &  \cO_{\PP^2}^{\alpha_{1,1}} \ar[r] &  0\\
    &\cO_{\PP^2}(1)^{\alpha_{1,2}} \ar@{=}[r] \ar[d]&\cO_{\PP^2}(1)^{\alpha_{1,2}}\ar[d]\\
    & 0 & 0
  }\]
  From the leftmost column, it follows that $\XE^*$ has a resolution
  of the desired form, with $a=3\alpha_{1,2}-\alpha_{1,1}$ and $b=\alpha_{1,2}$. 
  Claim  \eqref{perforzasteiner} is thus proved.

\subsection{} \label{case}
  Let us now prove \eqref{quandoerigido}. We may assume that
  $\rk(E)>1$. We check that $E$
  being indecomposable and rigid forces $E$ to be simple, so that it
  has no endomorphism factoring through $\cO_{\PP^2}(t)$ and thus
  \eqref{perforzasteiner} applies.
  In the previous proof, we used this condition only for
  $\alpha_{1,0} = 0$, and \eqref{perforzasteiner} will apply if $\alpha_{0,3}=0$.
  
\subsection{}
We work with $\alpha_{1,0} = 0$.
  Let $e$ be the restricted map $e : \cO_{\PP^2}(-1)^{\alpha_{1,2}} \to
  \Omega_{\PP^2}(1)^{\alpha_{1,1}}$ extracted from $d_0$ and let $F$
  be its cone, shifted by $1$:
  \begin{equation}
    \label{Fagain}
  F \to \cO_{\PP^2}(-1)^{\alpha_{1,2}} \xr{e} \Omega_{\PP^2}(1)^{\alpha_{1,1}}.    
  \end{equation}
  This is a complex with two terms, and its cohomology is
  concentrated in degrees zero and one, namely
  $\cH^0 F \simeq \ker(e)$ and $\cH^1 F \simeq \coker(e)$.
  From \eqref{generale} we easily see that $F$ fits into a
  distinguished triangle:
  \begin{equation}
    \label{F}
  F \to \cO_{\PP^2}^{\alpha_{0,0}} \to E.
  \end{equation}
 
Applying $\Hom_{\PP^2}(\cO_{\PP^2},-)$ to \eqref{Fagain}, we get
  $\Ext_{\PP^2}^i(\cO_{\PP^2},F)=0$ for all $i$, so:
 \[
 \Ext_{\PP^2}^i(F,F) \simeq \Ext^i_{\PP^2}(E,F[1]), \qquad \mbox{for  all $i$}.
 \]
Also, we know that $\HH^2(\PP^2,E^*)=0$, so applying  $\Hom_{\PP^2}(E,-)$ to \eqref{F} we
 get:
 \[
 \Ext^1_{\PP^2}(E,E) \to \Ext^1_{\PP^2}(E,F[1]) \to 0.
 \]
 Putting this together, we obtain a surjection:
 \[
 \Ext^1_{\PP^2}(E,E) \epi \Ext^1_{\PP^2}(F,F) \simeq \Ext^1_{\Upsilon_3}(R,R),
 \]
 with $F \simeq \Phi(R)$.
 We understand now that, if $E$ is rigid, then also $R$ is.

\subsection{}

If $R$ is rigid, then by \S \ref{KAC}, $R$ is a direct sum of
rigid simple representations of the form $R_k$. 
Therefore, cohomology of \eqref{F} gives an exact
sequence:
\[
0 \to \oplus_{i\le 0} F^{r_i}_{i} \to \cO_{\PP^2}^{\alpha_{0,0}} \to E \to
\oplus_{i \ge 1} F^{r_i}_i \to 0,
\]
for some integers $r_i$.
If only $R_i$ with $i\le 0$ appear, then we are done by \S \ref{case}.
Indeed, in that case
$E$ is globally generated, so $\HH^0(\PP^2,E) \ne 0$
implies $\HH^0(\PP^2,E^*)=0$ for otherwise $\cO_{\PP^2}$ would be a
direct summand of $E$.

If some $R_i$ appears with $i \ge 1$, we call $I$ the (non-zero) image
of the middle map in the previous exact sequence, and we check
$\Ext^1_{\PP^2}(F_j,I)=0$ for all $j \ge 1$, which contradicts $E$ being
indecomposable.
To check this, note that $\Hom_{\PP^2}(F_j,-)$ gives an exact
sequence:
\[
\Ext^1_{\PP^2}(F_j,\cO_{\PP^2})^{\alpha_{0,0}} \to \Ext^1_{\PP^2}(F_j,I) \to \oplus_{i\le 0}\Ext^2_{\PP^2}(F_j,F_i)^{r_i}.
\]
The leftmost term vanishes by Serre duality and Lemma \ref{cohom}.
The rightmost term is zero by Lemma \ref{reciproci}.
Part \eqref{quandoerigido} is now proved.

\subsection{} The statement \eqref{nienteH1}
 is clear by Lemma \ref{cohom} and by exceptionality of Fibonacci bundles.
The fact that $E^*$ is also ACM is obvious by Serre duality.

\begin{rmk}
  If $\bk$ is algebraically closed of characteristic zero, we may
  apply \cite[Corollaire 7]{drezet:beilinson}, to the effect that a
  rigid bundle
  is a direct sum of exceptional bundles.
  So, at the price of relying on this result,
  from \eqref{perforzasteiner} we may deduce directly
  \eqref{quandoerigido} via Kac's theorem.
\end{rmk}


\section{ACM bundles on the second Veronese threefold}

The techniques we have just seen apply to the embedding of $\PP^3$ in
$\PP^9$ by quadratic forms.
Again we replace $E$ with the $E(s)$, where $s$ is the smallest integer such
that $E$ has non-zero global sections, and set
$\alpha_{i,j}=\hh^i(\PP^3,\XE(-j))$.
If \eqref{sonmalato} gives $\alpha_{1,1}=\alpha_{2,1}=0$, then
$\XE(-1)$ has the desired resolution.
On the other hand, if \eqref{sonmalato} tells
$\alpha_{1,0}=\alpha_{2,0}=\alpha_{1,2}=\alpha_{2,2}=0$, then we are
left with a resolution of the form:
\[
0 \to    \cO_{\PP^3}(-1)^{\alpha_{1,3}} \xr{d_0} \Omega_{\PP^3}(1)^{\alpha_{1,1}}
\oplus  \cO_{\PP^3}^{\alpha_{0,0}} \to \XE \to 0.
\]
This time we also have $\alpha_{0,4}=0$, and $\alpha_{1,4}=\alpha_{2,4}=0$ again by
\eqref{sonmalato}, and simplicity of $E$ gives
$\alpha_{3,4}=\hh^0(\PP^3,\XE^*)=0$. 
So $\XE(-1)$ has a resolution like:
\[
0 \to \XE(-1) \to \Omega_{\PP^3}^2(2)^{\alpha_{1,3}} \to \cO_{\PP^3}^{\alpha_{1,1}} \to 0.
\]
Then, using the same trick as in the proof of the previous theorem,
we see that $\XE^*$ has the desired resolution, with
$a=6\alpha_{1,3}-\alpha_{1,1}$ and $b=\alpha_{1,3}$.
 
  This proves the first statement. The rest follows by the same path.
  Drezet's theorem as shortcut for $\eqref{perforzasteiner}
  \Rightarrow \eqref{quandoerigido}$ may be replaced by
  \cite{happel-zacharia:self-extensions}.

\begin{rmk}
It should be noted that, in \cite[Theorem 1.2 and Theorem
1.3]{iyama-yoshino}, the ACM bundle $E$ on the given Veronese variety is assumed to  
have a rigid module of global sections. This implies, respectively,
$\Ext^1_{\PP^2}(E,E(3t))=0$, or $\Ext^1_{\PP^3}(E,E(2t))=0$, for all
$t \in \ZZ$.
A priori, this is a stronger requirement than just
$\Ext^1_{\PP^n}(E,E)=0$.
However, our proof shows that the two conditions are  equivalent for
ACM bundles.
\end{rmk}

\section{Rigid ACM bundles higher Veronese surfaces}

Assume $\bk$ algebraically closed.
The next result shows that, for $d \ge 4$, the class of rigid ACM bundles
on $d$-fold Veronese surfaces contains the class
of exceptional bundles on $\PP^2$, which is indeed quite
complicated, cf. \cite{drezet-lepotier:stables}.
At least if $\mathrm{char}(\bk)=0$, the two classes coincide by
\cite[Corollaire 7]{drezet:beilinson}.

\begin{thm}
  Let $F$ be an exceptional bundle on $\PP^2$ and fix $d \ge 4$.
  Then there is an integer $t$ such that $E=F(t)$ satisfies \eqref{sonmalato}.
\end{thm}

\begin{proof}
  It is known that $F$ is actually stable by \cite{drezet-lepotier:stables}.
  This implies that $F$ has natural cohomology by
  \cite{hirschowitz-laszlo:generiques}, i.e. for all $t \in \ZZ$ there
  is at most one $i$ such that $\HH^i(\PP^2,F(t)) \ne 0$.
  Then, $\HH^1(\PP^2,F(t)) \ne 0$ if and only if $\chi(F(t))<0$.

  Let now that $r$, $c_1$ and $c_2$ be the rank and the Chern classes
  of $F$. Computing $\chi$ by additivity, we see that $\chi(F(t))$ is a polynomial of
  degree $2$ in $t$, of dominant term $r/2$, whose discriminant is:
  \[
  \Delta = {c}_{1}^{2}(1-r)+r (2{c}_{2}+r/4 ) = -\chi(F,F)+5r^2/4.
  \]

  So, using $\chi(F,F)=1$, we get $\Delta = -1+5r^2/4$.
  Therefore, the roots of $\chi(F(t))$ differ at most by:
  \[
  \mbox{$\lceil \frac{2\sqrt{\Delta}}{r} \rceil = \lceil\frac{\sqrt{5r^2-4}}{r} \rceil \le 3$}.
  \]
  Then, there are at most three consecutive integers $t_0$, $t_0+1$, $t_0+2$ such that
  $\HH^1(\PP^2,F(t_0+j))\ne 0$ for $j=0,1,2$. This means that
  $E=F(t_0-1)$ satisfies \eqref{sonmalato} for any choice of $d \ge 4$.
\end{proof}

\noindent {\bf Acknowledgements}. I would like to thank
F.-O. Schreyer, J. Pons Llopis and M. C. Brambilla for useful comments
and discussions.

\bibliographystyle{alpha-my}
\bibliography{bibliography}


\vfill

\end{document}